\documentclass[11pt,a4paper,reqno]{amsart}

\usepackage{amsopn, amssymb, amscd, amsmath, amsfonts, amsthm,cases}
\usepackage{verbatim}
\usepackage[english]{babel}

\newcommand{\ma}{\mathfrak a}

\newcommand{\mg}{\mathfrak g}
\newcommand{\mh}{\mathfrak h}

\newcommand{\mm}{\mathfrak m}
\newcommand{\mn}{\mathfrak n}
\newcommand{\mpp}{\mathfrak p}

\newcommand{\mv}{\mathfrak v}
\newcommand{\mz}{\mathfrak z}
\newcommand{\mso}{\mathfrak{so}}
\newcommand{\msu}{\mathfrak{su}}
\newcommand{\msp}{\mathfrak{sp}}
\newcommand{\msl}{\mathfrak{sl}}

\newcommand{\HH}{\mathbb H}
\newcommand{\RR}{\mathbb R}
\newcommand{\CC}{\mathbb C}
\newcommand{\OO}{\mathbb O}

\newtheorem{Theorem}{Theorem}

\newtheorem{Proposition}{Proposition}
\newtheorem{Corollary}{Corollary}

\theoremstyle{remark}
\newtheorem{Remark}{Remark}

\begin{document}

\title[Parabolic nilradicals of Heisenberg type, II]{Parabolic nilradicals of Heisenberg type, II}
\author{Aroldo Kaplan}
\author{Mauro Subils}
\date{\today}
\address{A. Kaplan: CIEM-CONICET, Famaf, U.N.C., Cordoba 5000, Argentina, and Department of Mathematics, U. of Massachusetts, Amherst, MA 01002, USA}
\email{kaplan@math.umass.edu}
\address{M. Subils: CONICET-FCEIA, U.N.R., Pellegrini 250, 2000 Rosario, Argentina.} \email{msubils@gmail.com}

\thanks{This work was partially supported by grants PICT from CONICET for both authors, and a post-doctoral Fellowship from CONICET and a grant from ANPCyT-FONCyT for the second author.}

\maketitle


\maketitle


This is a sequel to \cite{KS}, where we proved that every real simple non-compact Lie algebra different from $\mso(1,n)$ has an essentially unique parabolic subalgebra whose nilradical is a  Heisenberg algebra of a division algebra, and deduced some consequences regarding their Tanaka prolongations.

 Here we discuss their symmetries, the associated parabolic geometries,  and the riemannian geometry of the harmonic spaces $\mathbb R_+ N$, having the former as conformal infinities.

It is somewhat remarkable that such basic result had not been noticed before. Since \cite{KS} was written we found that it can also be deduced from \cite{Ho}\cite{W}\cite{Sa}\cite{Ch} by identifying the building blocks of Howe's  H-tower groups\footnote{Table 1 in \cite{Ho} is basically the same as ours}. Still, our construction is independent of type H and representation theory, and explains part of  the ``high degree of symmetry'' observed in \cite{Ho}.

We thank E. Hullet, A. Tiraboschi and J. Vargas for fruitful conversations.

\

 {\bf 1. Algebras of type divH and associated parabolics}

\

Recall that a normed real division algebra $\mathbb A$ determines two series of graded nilpotent algebras  $$\mathfrak h_n(\mathbb A)=(\mathbb A^{n}\oplus\mathbb A^{n})\oplus \mathbb A,\qquad\mathfrak h'_{p,q}(\mathbb A)= (\mathbb A^{p}\oplus\mathbb A^{q})\oplus \Im(\mathbb A)$$ with respective brackets $$[ x +y+t,\hat x  +\hat y+\hat t]= \sum x_i\hat y_i - \hat x_iy_i$$ $$[x +y+t,\hat x +\hat y+\hat t]= -\Im(\sum_j x_j\bar {\hat x}_j + \sum_k \bar {\hat y}_ky_k).$$Excluding  the $\mh_{p,q}'(\RR)$'s, which are abelian, and the $\mh_{n}(\OO)$, $\mh_{p,q}'(\OO)$ for $n>1$ or $p+q>1$, which are non-prolongable (see 3.2 below), and taking the isomorphism $\mathfrak h'_{p,q}(\mathbb C)\cong \mathfrak h_{p+q}(\mathbb R)$ into account, the remaining ones are $$\mh_n(\RR)\qquad\mh_n(\CC)\qquad\mh_{n}(\HH)\qquad\mh_{p,q}(\HH)\qquad\mh_{1}(\OO)\qquad\mh_{1,0}(\OO).$$ We call these algebras and associated objects  \emph{of type $divH$},  or when convenient, $\mathfrak h(\mathbb A)$.

Let now $\mg$ be a simple Lie algebra, $\mpp\subset \mg$ a  parabolic subalgebra, and $\mn\subset\mpp$ the nilradical of $\mpp$. If $\mg$ is compact, then $\mpp$ is either $0$ or $\mg$. If $\mpp$ is proper
and $\mg$ is isomorphic to $\mso(1,n)$, then $\mpp$ is unique up to conjugacy and $\mn$ is abelian. Moreover, the $\mso(1,n)$ are the only simple algebras with these properties.
Here we will be interested in the remaining ones, those which contain parabolic subalgebras with non-abelian nilradical, the set of which we which often denote by $\mathfrak S$.

For a graded nilpotent Lie algebra $\mn=\mg^{-1}\oplus\mg^{-2}\oplus \ldots$ to be the nilradical of a parabolic subalgebra of a semisimple algebra is equivalent to asking that it can be ``prolonged'' to a finite dimensional graded semisimple algebra
\begin{equation}\label{eq:prol}
\mg(\mn,\mg_0)=\mg^{k}\oplus\ldots \oplus \mg^{1}\oplus\mg^0\oplus\mg^{-1}\oplus \ldots\oplus\mg^{-k}
\end{equation}
where $\mg^{i}=\theta\mg^{-i}$ for some Cartan involution. This already  implies that $\mathrm{Aut}(\mn)$ must be large enough, so as to contain such  $\mg^0$. The associated parabolic subalgebra is $\mg^0\oplus\mn$.

The main results of \cite{KS} can be resumed as follows.

\begin{Theorem}
	\label{teo:main}\cite{KS}

	\begin{enumerate}
		
		\item[(a)] Every simple non-compact Lie algebra not isomorphic to $\mso(1,n)$ has a parabolic subalgebra with non-singular nilradical.
		
		\item[(b)] Any two are conjugate by the adjoint group.
		
		\item[(c)] The nilradicals that appear are exactly the algebras of type $divH$.
		
		\item [(d)] An algebra of type H is of type divH if and only its Tanaka prolongation is not trivial.
	\end{enumerate}
	
\end{Theorem}

One consequence is that divH algebras are the most symmetric among 2-step non-singular nilpotents, in the following sense.
Let  $\mn= \mv\oplus \mz$
be a 2-graded nilpotent Lie algebra
with center $\mz$ and let $m=\dim\mz$ and $n=\dim\mv$.
Since
$\mathrm {Der}(\mathfrak n)= \mathrm{Der_{gr}}(\mathfrak n)\oplus
\mathrm{Hom}(\mv,\mz),$
$$\dim\mathrm {Der}(\mathfrak n)=\dim\mathrm{Der_{gr}}(\mathfrak n)+mn.$$
Generically,
$\dim\mathrm{Der_{gr}}(\mn)=1$. If $\mn$ is of type H,
$$\dim\mathrm{Der_{gr}}(\mn)\geq \frac{1}{2}m(m+1).$$

Now let $N$ be the csc Lie group with  Lie algebra $\mn$, $\mathcal V$ the left-invariant distribution on $N$ determined by $\mv$, and
$\mathrm{Inf}(\mn)$ the algebra of infinitesimal automorphisms of
$\mathcal V$ at $e$, that is,  germs of vector fields $X$ on $N$ near $e$ such that $L_X(\mathcal V)\subset \mathcal V$. Clearly, $\mathrm{Inf}(\mn)\supset \mathrm{Der_{gr}}(\mathfrak n)$. Then, generically, even among type H,

$$\dim\mathrm{Inf}(\mn)/ \mathrm{Der_{gr}}(\mathfrak n)=\dim\mn.$$
For type $divH$ instead,
$$\dim\mathrm{Inf}(\mn)/\mathrm{Der_{gr}}(\mathfrak n)\geq 2\dim\mn.$$

\

{\bf 2. Langlands decompositions}

\

 Let $\mn$ be of type $divH$, $\mpp$ a standard parabolic subalgebra of some simple Lie algebra $\mg$ having $\mn$ as nilradical, and
	$$\mpp = \mm \oplus \ma\oplus\mn$$
	its Langlands decomposition. Then
	
	\begin{Proposition}
	$$\mm =\mm_o \oplus \mathfrak{spin}(\mn),\qquad \ma=\ma_o\oplus\ma_\delta,\qquad \mg_0 =\mm_o\oplus \ {\mathfrak {spin}}(\mn)\oplus\ma$$
	where
	
	$\mm_o$ is the centralizer of $\mz$ in $\mm$;
	
	${\mathfrak{spin}}(\mn)\cong\mso(\mz)$ acts on $\mz$ by the  standard representation and on $\mv$ as a sum of spin representations;
	
	$\ma_o=\ma\cap\mathrm{Der}_o(\mn)$, which  is $0$ if $\mpp$ is maximal and  1-dimensional otherwise;

    and $\ma_\delta=\RR\delta$.
	
	\noindent The individual factors of  the resulting decomposition
	$$\mpp=(\mm_o\oplus \ma_o)\oplus\mathfrak{spin}(\mn) \oplus (\ma_\delta \oplus\ \mn)$$
	are listed in Table 1.
\end{Proposition}
\begin{proof}
	All the assertions follow from  Table 1, which is obtained applying the construction in the proof of Theorem 1 case by case.
\end{proof}

Given a simple $\mg\in\mathfrak S$, denote by $\mpp(\mg)$ the parabolic subalgebra with nilradical $\mn(\mg)$ of type
$divH$, and $\mm(\mg)$ its Levi factor.

\begin{Corollary}
	$\mg\in \mathfrak S$ has a complex or quaternionic structure  if and only if $\mn(\mg)$ is  $\mathfrak  h_n(\mathbb C)$ or $\mathfrak  h_n(\mathbb H)$, respectively.
\end{Corollary}

\begin{Proposition}
	$\mpp(\mg)$ is maximal parabolic except for $\mg\cong$
	$\msl(n,\RR)$, $\msl(n,\CC)$, $\msu^{*}(2n)$, and $EIV$.
	It is minimal iff $\mg\cong$
	$\msu(1,n)$, $\msp(1,n)$, $\msu^{*}(6)$, $FII$, $\msl(3, \RR)$, $\msl(3,\CC)$, or $EIV$.
\end{Proposition}

Even if $\mpp(\mg)$ is not minimal, it contains the following distinguished minimal one. First note that any reductive Lie algebra can be uniquely decomposed as
$ \mathfrak r= \mathfrak r'\oplus\mathfrak r''$ where $\mathfrak r'$ is semisimple with simple factors in $\mathfrak S$, and $\mathfrak r''$ is reductive with simple factors not in $\mathfrak S$.

\begin{Proposition}
	
	\item[(a)] If $\mg\in\mathfrak S$, then
	$\mm(\mg)'\in\mathfrak S$.
	\item[(b)] If $\mg\in\mathfrak S$ is classical, then
	$\mm(\mg)'$ is classical and of the same type as $\mg$.
	\item[(c)] If $\mn$ is  $divH$, then $\mathrm{Der_o}(\mn)'\in\mathfrak S$.
\end{Proposition}
\begin{proof}
	By inspection of Table 1
\end{proof}

One obtains a filtration
$$\mg=\mg^0\supset\mg^{-1}\supset\ldots\supset\mg^{-k}$$
of Lie subalgebras
all of  class $\mathfrak S$, with corresponding $divH$-nilradicals $\mn(\mg^{-i})$, such that $ \mg^{-i-1} =   \mm(\mg^{-i})'$.

\begin{Proposition}
	$\mpp(\mg^{-k})$ is a minimal parabolic subalgebra of $\mg$, and $\bigoplus_{i=0}^k\mn(\mg^{-i})$ is a maximal nilpotent one.
\end{Proposition}

It follows that every classical $\mg\in\mathfrak S$ fits into a strictly increasing filtration of algebras
$$0\subset\mathfrak g^{-k}\subset \ldots\subset \mathfrak g^{-1}\subset \mathfrak g \subset \mathfrak g^{1} \subset  \ldots$$
 of simple algebras of the same simple type and satisfying
$\mg^{i-1} = (\mg^{i})'$.

\begin{Remark} The $\bigoplus_{i=0}^j\mn(\mg^{-i})$ are essentially Howe's H-tower algebras.
\end{Remark}

\

{\bf 3. Parabolic Geometries}

\

Among the distributions with symbol of type H,  those with symbol of type $divH$ have compact Klein models. More precisely, let $\mathcal V$ be the canonical distribution on  a group $N$ of type $divH$, and choose a simple $G$ and a parabolic $P$ with $N$ as nilradical. The tangent space to $G/P$ at the origin can be identified with $\bar \mn=\bar \mv\oplus\bar \mz$, and $P$ respects this grading. Let $\bar{\mathcal V}$ be the $G$-invariant distribution on $G/P$ determined by $\bar\mv$. Therefore

\begin{Proposition} $G/P$ carries a $G$-invariant distribution locally equivalent to $\mathcal V$.\end{Proposition}

Consider now a parabolic geometry of type $divH$ on a manifold $M$, i.e., a Cartan geometry of type $(G,P)$ where $\mg=Lie(G)\in\mathfrak S$ and $P\subset G$ is a parabolic subgroup with unipotent radical of type $divH$. Let $\omega$ be its Cartan connection and $\kappa$ its curvature. Together with the gradings
$$\mg=\mg_{-2}\oplus  \mg_{-1}\oplus \mg_{0}\oplus \mg_{1}\oplus\mg_{2}$$
$$ \mpp=\mg_{0}\oplus \mg_{1}\oplus\mg_{2}  \qquad \mn=\mg_{-2}\oplus \mg_{-1},$$
$\omega$ determines a distribution
$\mathcal D$ on $M$ and a principal $G_{0}$-bundle, where $G_{0}$ is the subgroup of $P$ that preserves the grading of $\mg$. $G_{0}$ is isomorphic to a subgroup of $\mathrm{Aut}_{gr}(\mn)$ and $Lie(G_{0})=\mg_{0}$.

$\omega$ is called \emph{regular} when $\mathcal D$ has constant symbol isomorphic to $\mn$ and the principal $G_{0}$-bundle is a reduction of the canonical $\mathrm{Aut}_{gr}(\mn)$-bundle. When $G_{0}=\mathrm{Aut}_{gr}(\mn)$, we have just a distribution of constant symbol.
$\omega$ is called \emph{normal} if it satisfies $\partial^* \kappa=0$ where $\partial^*$ is the Kostant codifferential. This condition assures the uniqueness of the Cartan connection.

Since a distribution is fat if and only if its symbol is non-singular, Theorem 1 (a) implies

\begin{Theorem} The regular normal parabolic geometries supported on fat distributions are exactly those of $divH$ type.
\end{Theorem}
\noindent In fact,

(a) Distributions with symbol $\mathfrak h_{n}(\mathbb R)$ are associated to contact parabolic geometries: Lagrangean, partially integrable almost CR, Lie contact, contact projective, and exotic contact structures \cite{CS}.

(b)Distributions with symbol $\mathfrak h_{n}(\mathbb C)$ are associated to the complex contact structures of Boothby \cite{Bo} and, more generally, to partially integrable almost CR-structures of CR-codimension $2$ with additional structure. These have not received much attention except for special cases \cite{CSc}.
			
(c) For the cases $\mg = \msp(n, \mathbb{R})$, $\msp(n, \mathbb{C})$, $\mg$ is not the prolongation of $(\mn,\mg_{0})$, so the underlying structure they determine on the manifold is just a real or complex contact structure with the canonical $\mathrm{Aut}_{gr}(\mn)$-bundle. To characterize this parabolic geometries we have to consider finer underlying structures, in this case are a contact projective structure, i.e. a contact projective equivalence class of partial contact connections \cite{CS}.
			
(d) For all the other $divH$ algebras the parabolic geometry is determined by the distribution alone, with no additional structure (Proposition 4.3.1 in \cite{CS}). Quaternionic and octonionic contact structures associated to $\mh_{p,q}'(\HH)$ and $\mh_{1,0}'(\OO)$ have been the subject of interest \cite{Bi, CS}.
			
(e) Distributions whose symbol is $\mathfrak h_{1}(\mathbb O)$ or $\mathfrak h_{1,0}'(\mathbb O)$ are locally isomorphic to the flat model. This is a consequence of the fact
that the second generalized Spencer cohomology groups vanish in these cases \cite{Y}.

\

{\bf 1. Conformal infinity of harmonic spaces}

\

Let $N$ be a group of type H, $A=\exp(\RR\delta)$ the group of dilations and $S=AN$ their semidirect product. Endowing $\mn$ with a compatible  metric induces a left-invariant riemannian metric $g$ on $S$, called a Damek-Ricci metric \cite{R}.
$S$ is harmonic - hence Einstein, and any  homogeneous non-compact harmonic space is isometric to $S$ for some $N$ of type H \cite{Heb1}. One is interested in the asymptotic behavior of the metric $g$.

 If $S$ is hyperbolic, the metric in polar form satisfies
\begin{equation}\label{eq:met}
g= dt^2 + e^t\gamma+e^{2t}\delta + o(t)
\end{equation}
for $t\to\infty$, where $(\gamma,\mathcal D^{\delta})$ is a generalized G-conformal structure in the sense of Biquard-Mazzeo \cite{BM} on  the geodesic boundary of $S$, which for the hyperbolic space is a sphere.

For the general $S$ no such formula seems to exist (cf. for example \cite{C}) -  unless $\mn$ is of type $divH$. For the first statement, consider the Poincar\'e-like realization of $S$ in euclidean unit ball $\mathbb B$ of the same dimension, as well as the Siegel-like one on
$$\mathcal U = \{( X,Z,t)\in \mathfrak{v\times z\times \RR}:\ \ t>\frac{1}{4}|X|^2\}.$$
The Cayley transform $\mathcal C:\mathcal U\to\mathbb B$
$$\mathcal C(X,Z,t)= \frac{1}{(1+t)^2+|Z|^2}\big( (1+t-J_{Z})X, 2Z, -1+t^2-|Z|^2\big)$$
is a diffeomorphism. It extends to the boundary of $\mathcal U$ in $v\times z\times \RR$,
 $\partial\mathcal U = \{(X,Z,\frac{1}{4}|X|^2): X\in\mathfrak v,\, Z\in \mathfrak z\}$
giving a diffeomorphism
$$\mathcal C_{\partial}:\partial\mathcal U\to\mathbb S^*$$
onto the punctured sphere.
$N$ acts simply transitively on $\partial\mathcal U$, hence on $\mathbb S^*$, and its  canonical distribution  induces invariant distributions on these boundaries. Writing
$$T_{(X ,Z ,\frac{1}{4}|X |^2)}(\partial\mathcal U) =\{(2Y,W, <X ,Y>): Y\in \mathfrak v,\  W\in \mathfrak z\},$$
the distribution is given by
$${\mathcal D}_{(X,Z,\frac{1}{4}|X|^2)}^{\partial\mathcal U} =  \{(Y, \frac{1}{2}[X,Y],\frac{1}{2}<X,Y>): Y\in \mathfrak v\}.$$
Let now ${\mathcal D}^{\mathbb S^*}= d\mathcal C_{\partial}({\mathcal D}^{\partial\mathcal U})$ and let $\infty$ denote the puncture of $\mathbb S^*$.

\begin{Proposition}
	${\mathcal D}^{\mathbb S^*}$ extends smoothly over $\infty$ if and only if $S$ is a hyperbolic space.
\end{Proposition}

\begin{proof}
	If $S$ is a hyperbolic space $G/K$, $K$ is transitive on $\mathbb S$ and leaves invariant the distribution, hence it can have no singularities.
	
	Otherwise, $N$ does not satisfy the $J^2$ condition of \cite{CDKR}. This implies that there is a unitary triple $X\in\mv$, $Z,W\in\mz$ such that $[X,J_{Z}J_{W}X]=0$.
The vector fields  on $\partial\mathcal U\cong\mv\times\mz$
$$(v_1)_{(X,Z)}=(X, 0,\frac{1}{2}|X|),\qquad (v_2)_{(X,Z)}=(J_{Z}X, \frac{1}{2}|X|Z,0)$$
	correspond to the copy of $\mh_1(\RR)$ spanned by the triple $X,J_{Z} X, Z$. On $\mathbb S^*$ and along the  orbit $\exp(\mh_1(\RR))\cdot (-\infty)$,   the plane spanned by $d\mathcal C(v_1)$,  $d\mathcal C(v_2)$,  is horizontal and has a limit as $|X|,|Z|\to\infty$, namely the plane $(\RR X\oplus\RR J_{Z} X,0,0)$. Doing the same with the copy of  $\mh_1(\RR)$ spanned by $J_{Z}J_{W}X, J_{W}  X, Z$, the corresponding limiting plane is $(\RR J_{W}X\oplus\RR J_{Z}J_{W} X,0,0)$. Therefore, if the distribution extends, its value at $\infty$ must be $(\mv,0,0)$.
	On the other hand, the vector
		$$ ((1+ |Z|^2-\frac{1}{16}|X|^4 ) J_{W}X, (1+\frac{1}{4}|X|^2)|X|^2W,0)$$
is horizontal along the curve $1+ |Z|^2=\frac{1}{16} |X|^4$, where it spans line $(0,\RR W,0)$, which is a contradiction.
\end{proof}

If $N$ is of type $divH$ however, the $S$-orbit of any point gives an isometric embedding into the associated symmetric space
$$S\hookrightarrow G/K.$$
Denoting  by $\partial S$ the boundary of $S$ in an appropriate compactification of $G/K$, the natural projection
$$\pi: \partial S\to \mathbb S$$
onto the geodesic spherical boundary resolves the singularities of $\mathcal D^{\mathbb S^*}$.

As a consequence, (\ref{eq:met})  holds in this case.
Indeed such formula seems to characterize the $divH$ among  harmonic spaces. The non-hyperbolic ones are those obtained for $\mn= \mh_{n}(\CC), \mh_{p,q}'(\HH) \ (pq\not=0), \mh_{n}(\HH), \mh_{1}(\OO)$, are all anisotropic, and the last three admit non-regular deformations, suited to extend the arguments of \cite{BM} to obtain new Einstein metrics.
Details are left for a sequel, where the boundary structures $(\gamma,\delta)$ will be described for each  $divH$ type.

\begin{table}
	\caption{Langlands factors of $divH$ parabolics}
	\begin{center}
		\begin{tabular}{c c c c c}
			\hline\noalign{\smallskip}
			$\mg$	&  $\mm$ &  dim$\ma$ &  $\mn$	 & $\Sigma$	  \\
			\noalign{\smallskip}\hline\noalign{\smallskip}
			$\msl(n,\RR)$ &  $\msl(n-2,\RR)$ & $2$ & $\mh_{n-2}(\RR)$ &  $\{\alpha_{1}, \alpha_{n-1}\}$   \\
			$\msl(n,\CC)$  & $\msl(n-2,\CC)\oplus\RR^{2}$  & $2$ & $\mh_{n-2}(\CC)$  & $\{\alpha_{1}, \alpha_{n-1}\}$  \\
			$\msu^{*}(2n)$ & $\msu(2)^{2}\oplus\msu^{*}(2n-4)$ & $2$ & $\mh_{n-2}(\HH)$ & $\{\alpha_{2}, \alpha_{n-2}\}$  \\
			
			$\msu(p,q)$ &  $\msu(p-1,q-1)\oplus\RR$ & $1$ & $\mh_{p+q-2}(\RR)$ &  $\{\alpha_{1}, \alpha_{p+q-1}\}$  \\
			
			$\msp(n,\RR)$ &  $\msp(n-1,\RR)$  & $1$  & $\mh_{n-1}(\RR)$ &  $\{\alpha_{1}\}$ \\
			$\msp(p,q)$  & $\msu(2)\oplus\msp(p-1,q-1)$ & $1$ & $\mh'_{p-1,q-1}(\HH)$ & $\{\alpha_{2}\}$ \\
			$\msp(n,\CC)$  & $\msp(n-1,\CC)\oplus\RR$ & $1$ & $\mh_{n-1}(\CC)$ &  $\{\alpha_{1}\}$  \\
			$\mso(p,q)$  &  $\msl(2,\RR)\oplus\mso(p - 2,q - 2)$ & $1$ & $\mh_{p+q - 4}(\RR)$ &  $\{\alpha_{2}\}$  \\
			$\mso^{*}(2n)$  & $\msu(2)\oplus\mso^{*}(2n-4)$ & $1$ & $\mh_{2n-4}(\RR
			)$   & $\{\alpha_{2}\}$  \\
			$\mso(n,\CC)$  & $\msl(2,\CC)\oplus\mso(n-4,\CC)\oplus\RR$ & $1$ & $\mh_{n-4}(\CC)$ & $\{\alpha_{2}\}$\\
			$EI$& $\msl(6,\RR)$&$1$&$\mh_{10}(\RR)$&$\{\alpha_{6}\}$\\
			$EII$&$\msu(3,3)$&$1$&$\mh_{10}(\RR)$&$\{\alpha_{6}\}$\\
			$EIII$&$\msu(1,5)$&$1$&$\mh_{10}(\RR)$&$\{\alpha_{6}\}$\\
			$EIV$&$\mso(8)$&${2}$&$\mh_{1}(\OO)$&$\{\alpha_{1}, \alpha_{5}\}$\\
			$E_{6}$&$\msl(6,\CC)\oplus\RR$&$1$&$\mh_{10}(\CC)$&$\{\alpha_{6}\}$\\
			$EV$&$\mso(6,6)$&$1$&$\mh_{16}(\RR)$&$\{\alpha_{6}\}$\\
			$EVI$&$\mso^{*}(12)$&$1$&$\mh_{16}(\RR)$&$\{\alpha_{6}\}$\\
			$EVII$&$\mso(2,10)$&$1$&$\mh_{16}(\RR)$&$\{\alpha_{6}\}$\\
			$E_{7}$&$\mso(12,\CC)\oplus\RR$&$1$&$\mh_{16}(\CC)$&$\{\alpha_{6}\}$\\
			$EVIII$&$EV$&$1$&$\mh_{28}(\RR)$&$\{\alpha_{1}\}$\\
			$EIX$&$EVII$&$1$&$\mh_{28}(\RR)$&$\{\alpha_{1}\}$\\
			$E_{8}$&$E_{7}\oplus\RR$&$1$&$\mh_{28}(\CC)$&$\{\alpha_{1}\}$\\
			$FI$&$\msp(3,\RR)$&$1$&$\mh_{7}(\RR)$&$\{\alpha_{4}\}$\\
			$FII$&$\mso(7)$&$1$&$\mh'_{1,0}(\OO)$&$\{\alpha_{1}\}$\\
			$F_{4}$&$\msp(3,\CC)\oplus\RR$&$1$&$\mh_{7}(\CC)$&$\{\alpha_{4}\}$\\
			$G$&$\msl(2,\RR)$&$1$& $\mh_{2}(\RR)$&$\{\alpha_{2}\}$\\
			$G_{2}$&$\msl(2,\CC)\oplus\RR$&$1$&$\mh_{2}(\CC)$& $\{\alpha_{2}\}$\\
			\hline
			\hline\noalign{\smallskip}
		\end{tabular}
		\label{table:prol}
	\end{center}
\end{table}


\end{document}